\documentclass[10pt]{amsart}
 \usepackage[utf8]{inputenc} 
\usepackage{csquotes}

\usepackage{amsfonts}
\usepackage{amssymb}
\usepackage{graphicx,ulem,hyperref}
\usepackage{pstricks}
\usepackage{amsmath}
\usepackage{amsxtra}

\theoremstyle{plain}
\newtheorem{theorem}{Theorem}[section]
\newtheorem{lemma}[theorem]{Lemma}

\newtheorem{conjecture}[theorem]{Conjecture}

\newtheorem{problem}[theorem]{Problem}

\theoremstyle{definition}

\newtheorem{example}[theorem]{Example}
\newtheorem{remark}[theorem]{Remark}

\numberwithin{equation}{subsection} 

\begin{document}

\title[Hyponormal block Toeplitz operators with finite rank self-commutators]{Hyponormal block Toeplitz operators \\ with finite rank self-commutators}


\author[M. Abhinand]{Mankunikuzhiyil Abhinand}
\address{Department of Mathematics,University of Calicut,	Kerala-673635, India; and Indian Statistical Institute, Statistics and Mathematics Unit, 8th Mile, Mysore Road, Bangalore, 560059, India.}
\email{abhinandmkrishnan@gmail.com}

\author[R.E. Curto]{Ra\'ul E. Curto}
\address{Department of Mathematics, The University of Iowa, Iowa City, Iowa, U.S.A.}
\email{raul-curto@uiowa.edu}

\author[T. Prasad]{Thankarajan Prasad}
\address{Department of Mathematics, University of Calicut,	Kerala-673635, India}
\email{prasadvalapil@gmail.com}


\begin{abstract}
In this paper we identify a large class of hyponormal block Toeplitz operators whose self-commutators are of finite rank. \ Recall that an operator $T_\varphi$ is hyponormal and $[T_\varphi^{\ast}, T_\varphi]$ is a finite rank operator if and only if there exists a finite Blaschke product $b$ in $\mathcal{E}(\varphi)$, where 
$$
\mathcal{E}(\varphi) := \{k \in H^\infty(\mathbb{T}): \left\|k\right\|_\infty \le 1 \textrm{ and } \varphi-k\cdot \bar{\varphi} \in H^\infty(\mathbb{T}) \}.
$$
An analogous set $\mathcal{E}(\Phi)$ can be defined for a matrix-valued symbol $\Phi$. \ In the block Toeplitz operator case, we first establish that if a symbol $\Phi$ is in $L^\infty(\mathbb{T},M_n)$ and if $\mathcal{E}(\Phi)$ contains a constant unitary matrix $U$, then $T_\Phi$ is normal. \ We then obtain a suitable converse, under a mild assumption on the symbol. \ Next, we provide a partial answer to a conjecture recently posed by R.E. Curto, I.S. Hwang and W.Y. Lee \cite[Conjecture 6.1]{CHL}. \ Concretely, assume that $\Phi \in H^{\infty}(\mathbb{T}, M_n)$ is such that $\Phi^{\ast}$ is of bounded type and $T_\Phi$ is hyponormal. \ Then $[T_\Phi^{\ast}, T_\Phi]$ is a finite rank operator if and only if there exists a finite Blaschke–Potapov product in $\mathcal{E}(\widetilde{\Phi})$, where $ \widetilde\Phi:=\breve{\Phi}^*$ \; and \; $\breve{\Phi}(e^{i\theta}):=\Phi(e^{-i\theta})$. 

\vspace{5pt}

\textbf{Keywords}: Subnormal operators, Toeplitz operators, block Toeplitz operators, Blaschke-Potapov products, Nakazi-Takahashi Theorem.

\vspace{5pt}
\noindent
\textbf{2020 Mathematics Subject Classification:} Primary 47B20; Secondary 47A10.

\end{abstract}
\maketitle

\pagestyle{myheadings}


\section{Introduction}
In the setting of a Hilbert space $\mathcal{H}$, we denote by $\mathcal{B}(\mathcal{H})$  the algebra of all bounded linear operators on $\mathcal{H}$. \ An operator $T \in \mathcal{B}(\mathcal{H})$ is said to be \textit{self-adjoint} if it satisfies $T=T^{\ast},$ where $T^{\ast}$ is its adjoint. \  The self-commutator of $T$ is given by $[T^{\ast}, T] := T^{\ast} T - T T^{\ast}$. \ An operator $T$ is called \textit{normal} if $[T^{\ast}, T] = 0$, whereas it is called \textit{hyponormal} if $[T^{\ast}, T] \geq 0$. \ A particularly notable subclass of hyponormal operators is the class of \textit{subnormal operators}, which are those that can be extended to normal operators. \ Specifically, $T\in \mathcal{B}(\mathcal{H})$  is called \textit{subnormal} if there exists a Hilbert space $\mathcal{K} \supseteq \mathcal{H}$  and a normal operator $N$ on $\mathcal{B}(\mathcal{K})$ such that $N |_{\mathcal{H}} = T$  and $N \mathcal{H} \subseteq \mathcal{H}$. \ The class of subnormal operators, introduced by P.R. Halmos \cite{HAL0}, has been the focus of extensive research over many decades, providing deep insights into functional analysis and operator theory.

Research on normal and hyponormal operators has an extensive range of applications across various fields including mathematics, mathematical physics and machine learning (see \cite{IFA}, for instance). \ The theory of subnormal operators has found extensive applications in various fields, including analytic function theory, differential geometry, approximation theory and quantum mechanics \cite{HS, SZA}. \ In particular, subnormal operators with finite rank self-commutator have been one of the interesting topics in functional analysis and operator theory. \ The finite rank condition suggests that the operator is nearly normal, thereby simplifying the spectral analysis. \ In 1973, B.B. Morrel \cite{MOR} established that any subnormal operator whose self-commutator is of rank one can be written as a linear combination of the unilateral shift operator and the identity operator. \ Later in 1987, D. Xia \cite{XIA,XIA1} made an attempt to classify all subnormal operators with finite rank self-commutator. \ In 1998, J.B. Conway and L. Yang \cite{CY} posed the following open problem:

\begin{center}
{\it Classify all subnormal operators whose self-commutator has finite rank.}
\end{center}

In the major part of this work, we address this question in the framework of block Toeplitz operators in conjunction with a conjecture recently posed by R.E. Curto, I.S. Hwang and W.Y. Lee \cite[Conjecture 6.1]{CHL}.

Let $\mathbb{T}$ denote the unit circle in the complex plane. \ Let $L^2(\mathbb{T})$ and $L^{\infty}(\mathbb{T})$ denote the Hilbert space of all absolutely square--integrable functions and the space of all essentially bounded measurable functions, respectively, defined on $\mathbb{T}$ with respect to the normalized-Lebesgue measure. \ By $H^2(\mathbb{T})$ we mean the Hardy space corresponds to $L^2(\mathbb{T})$ and $H^{\infty}(\mathbb{T}): = L^{\infty}(\mathbb{T}) \bigcap H^2(\mathbb{T}).$ A function $\theta \in H^{\infty}(\mathbb{T})$ is called \textit{inner}
if $|\theta| = 1$ a.e. on $\mathbb{T}$. \ A function $\phi \in L^{\infty}$ is of \textit{bounded type} if $\phi = \frac{\phi_1}{\phi_2}$ for $\phi_1, \phi_2 \in H^{\infty}(\mathbb{T})$. \ If $\phi \in L^{\infty}(\mathbb{T})$, then the \textit{Toeplitz operator} with symbol $\phi$ on the Hardy space $H^2(\mathbb{T})$ is denoted as $T_\phi$ and is defined by
$$
T_\phi f := P(\phi f),
$$
where $f \in H^2(\mathbb{T})$ and $P$ is the orthogonal projection of $L^2(\mathbb{T})$ onto $H^2(\mathbb{T})$. \ A Toeplitz operator $T_\phi$ is called \textit{analytic} if $\phi \in H^{\infty}(\mathbb{T})$ and \textit{coanalytic} if $\overline{\phi} \in H^{\infty}(\mathbb{T}).$

The characterization of classes of Toeplitz operators based on the properties of their symbols is an interesting and exciting problem in operator theory. \ Normality of Toeplitz operators was characterized by A. Brown and P.R. Halmos \cite{BH}, while C. Cowen \cite{CC} gave a characterization for their hyponormality.

On the other hand, it seems quite difficult to determine the subnormality of Toeplitz operators based on the properties of their symbols. \  Moreover, P.R. Halmos \cite{HAL, HAL1} posed the following problem:
\begin{center}
\textit{``Is every subnormal Toeplitz operator either normal or analytic ?"}
\end{center}
In 1984, C. Cowen and J. Long \cite{CL} gave an example of a subnormal Toeplitz operator which is neither normal nor analytic and answered this problem negatively. \ Yet, a complete characterization of subnormal Toeplitz operators by the properties of their symbols still remains an open problem. \  This naturally gives rise to the following problems:
\begin{center}
\textit{``Which Toeplitz operators are subnormal?"} 
\end{center}

\begin{center}
\textit{``Which subnormal Toeplitz operators are either normal or analytic?"}
\end{center}
There are some partial answers to Halmos' problem as follows:

\vspace{0.25cm}
\noindent {\bf Abrahamse's Theorem} (\cite[Theorem]{AB}). \ Let
$\varphi\in L^\infty$ such that $\varphi$ or $\overline\varphi$ is
of bounded type. \  If $T_\varphi$ is hyponormal and $\ker
[T_\varphi^*, T_\varphi]$ is invariant under $T_\varphi$, then
$T_\varphi$ is normal or analytic. \ In particular, if
$T_\varphi$ is a subnormal operator then $T_\varphi$ is either
normal or analytic.

\vspace{0.25cm}
\noindent{\bf Nakazi-Takahashi Theorem} (\cite[Theorem 15]{NT}). \ If
$T_\varphi$ is subnormal and $\varphi = q \overline \varphi$, where $q$ is a finite Blaschke product, then $T_\varphi$  is either normal or analytic.

\vspace{0.25cm}
\noindent{\bf Extension of Nakazi-Takahashi Theorem} (\cite[Theorem 4.2]{ACHLP1}). \  Let $T_\varphi$ be subnormal and assume that $\varphi = q \overline{\varphi} + g \in L^{\infty}(\mathbb{T})$, where $q$ is finite Blaschke product and $g \in H^{\infty}(\mathbb{T})$. \ If $gH^2 \subseteq \text{Ran }T_\varphi,$ then $T_\varphi$ is normal or analytic.

\vspace{0.25cm}
In 1993, T. Nakazi and K. Takahashi characterized hyponormal Toeplitz operators with finite rank self-commutator as follows:

\vspace{0.25cm}
\noindent{\bf Hyponormal Toeplitz operators with finite rank self-commutator}(\cite[Theorem 10]{NT}). \ An operator $T_\varphi$ is hyponormal and $[T_\varphi^{\ast}, T_\varphi]$ is a finite rank operator if and only if there exists a finite Blaschke product $b$ in $\mathcal{E}(\varphi)$, where 
$$
\mathcal{E}(\varphi) := \{k \in H^\infty(\mathbb{T}): \left\|k\right\|_\infty \le 1 \textrm{ and } \varphi-k\cdot \bar{\varphi} \in H^\infty(\mathbb{T}) \}.
$$
We can then choose $b$ such that the degree of $b = \text{rank }[T_\varphi^{\ast}, T_\varphi].$

\vspace{0.25cm}
In this paper, we discuss the validity of the above theorem in the block Toeplitz operator setting, which is related to a conjecture recently posed by R.E. Curto, I.S. Hwang and W.Y. Lee \cite{CHL}. \ In the case of block Toeplitz operators, the appropriate generalization of $\mathcal{E}(\varphi)$ is
$$
\mathcal{E}(\Phi) := \{K \in H^\infty(\mathbb{T},M_n): \left\|K\right\|_\infty \le 1 \textrm{ and } \Phi-K\Phi^* \in H^\infty(\mathbb{T},M_n) \}.
$$
Let $I_n$ be the $n \times n$ identity matrix, $M_{n\times m}$
denote the set of all $n\times m$ complex matrices and $M_n\equiv
M_{n\times n}$. \  As denote by $L^2(\mathbb{T}, \mathbb{C}^n)$, $H^2(\mathbb{T}, \mathbb{C}^n)$ and $L^\infty(\mathbb{T}, \mathbb{C}^n)$ the set of all $\mathbb C^n$-valued
Lebesgue square integrable functions on $\mathbb T$, the associated
Hardy space and the set of all $M_n$-valued essentially bounded
functions on $\mathbb T$, respectively. \ For a function $\Phi\in
L^\infty(\mathbb{T}, \mathbb{C}^n)$, the block Toeplitz operator $T_\Phi$ with symbol
$\Phi$ is defined by $T_\Phi f :=P_n(\Phi f)$ for $f\in H^2(\mathbb{T}, \mathbb{C}^n)$, where $P_n$ is the orthogonal projection of $L^2(\mathbb{T}, \mathbb{C}^n)$ onto $H^2(\mathbb{T}, \mathbb{C}^n)$. \ Similarly, a block Hankel operator $H_\Phi$ with symbol $\Phi\in L^\infty(\mathbb{T}, \mathbb{C}^n)$ is defined by $H_\Phi
f:=J_n P_n^\perp(\Phi f)$ for $f\in H^2(\mathbb{T}, \mathbb{C}^n)$, where $P_n^{\perp}$ is the orthogonal projection of $L^2(\mathbb{T}, \mathbb{C}^n)$
onto $(H^2(\mathbb{T}, \mathbb{C}^n))^{\perp}\equiv L^2(\mathbb{T}, \mathbb{C}^n)\ominus
H^2(\mathbb{T}, \mathbb{C}^n)$ and $J_n$ is the unitary operator from $L^2(\mathbb{T}, \mathbb{C}^n)$ onto $L^2(\mathbb{T}, \mathbb{C}^n)$ defined by
$J_n(f)(e^{i\theta}):=e^{-i\theta}I_n f(e^{-i\theta})$ for $f\in L^2(\mathbb{T}, \mathbb{C}^n)$. \ If we write $H^2(\mathbb{T}, \mathbb{C}^n) = H^2(\mathbb T)\oplus\cdots\oplus H^2(\mathbb T)$, then we can easily see that
$$
T_\Phi=\begin{bmatrix}T_{\phi_{11}}&\hdots&T_{\phi_{1n}}\\
	&\vdots\\
	T_{\phi_{n1}}&\hdots&T_{\phi_{nn}}
\end{bmatrix}\quad\hbox{and}\quad
H_\Phi=\begin{bmatrix}H_{\phi_{11}}&\hdots&H_{\phi_{1n}}\\
	&\vdots\\
	H_{\phi_{n1}}&\hdots&H_{\phi_{nn}}
\end{bmatrix},
$$
where
$$
\Phi=\begin{bmatrix}{\phi_{11}}&\hdots& {\phi_{1n}}\\
	&\vdots\\
	{\phi_{n1}}&\hdots&{\phi_{nn}}
\end{bmatrix}.
$$
For a function $\Phi\in L^\infty(\mathbb{T}, M_{n\times m})$, write
$$
\breve{\Phi}(e^{i\theta}):=\Phi(e^{-i\theta}) \quad \hbox{and} \quad
\widetilde\Phi:=\breve{\Phi}^*.
$$
A function $\Phi\in L^\infty(\mathbb{T}, M_{n\times m})$ is of bounded type if each of its entries is of bounded type. \ A function $\Theta\in H^\infty(\mathbb{T}, M_{n\times m})$ is called an inner function if $\Theta^*\Theta=I_m$ a.e. on $\mathbb T$. \  The following basic properties for Toeplitz and Hankel operators are implicitly used in the sequel:
\begin{equation}\label{basic}
	\aligned
	&T_\Phi^*=T_{\Phi^*},\ \  H_\Phi^*= H_{\widetilde \Phi} \quad
	(\Phi\in
	L^\infty(\mathbb{T}, \mathbb{C}^n));\\
	&T_{\Phi\Psi}-T_\Phi T_\Psi = H_{\Phi^*}^*H_\Psi \quad
	(\Phi,\Psi\in L^\infty(\mathbb{T}, \mathbb{C}^n));\\
	&H_\Phi T_\Psi = H_{\Phi\Psi},\ \
	H_{\Psi\Phi}=T_{\widetilde{\Psi}}^*H_\Phi \quad (\Phi\in
	L^\infty(\mathbb{T}, \mathbb{C}^n), \Psi\in H^\infty(\mathbb{T}, \mathbb{C}^n));\\
	&H_\Phi^* H_\Phi - H_{\Theta \Phi}^* H_{\Theta\Phi} =H_\Phi^*
	H_{\Theta^*}H_{\Theta^*}^*H_\Phi \quad (\Theta\in H^\infty(\mathbb{T}, \mathbb{C}^n)
	\ \hbox{inner,}  \ \Phi\in L^\infty(\mathbb{T}, \mathbb{C}^n)).
	\endaligned
\end{equation}
For a function $\Phi\in L^2(\mathbb{T}, \mathbb{C}^n)$, we write

$$
\Phi_+:=\mathbb{P}_n (\Phi)\in H^2(\mathbb{T}, \mathbb{C}^n) \quad \hbox{and} \quad
\Phi_-:=[\mathbb{P}_n^{\perp} (\Phi)]^*\in zH^2(\mathbb{T}, \mathbb{C}^n),
$$
where $\mathbb{P}_n$ and $\mathbb{P}_n^{\perp}$ denote the
orthogonal projections form $L^2(\mathbb{T}, M_n)$ onto $H^2(\mathbb{T}, M_n)$ and
$(H^2(\mathbb{T}, M_n))^{\perp}\equiv L^2(\mathbb{T},M_n) \ominus H^2(\mathbb{T},M_n)$, respectively. \ Then we may write
$$
\Phi=\Phi_-^*+\Phi_+.
$$
For $\Phi\in L^\infty(\mathbb{T},M_n)$, the following identities hold:
$$
P_nJ_n=J_nP_n^\perp, \quad P_n^\perp J_n=J_n P_n  \quad \hbox{and}
\quad J_n M_{\Phi}=M_{\breve{\Phi}}J_n.
$$
An $n\times n$ matrix-valued function $Q$ is called a {\it finite
	Blaschke-Potapov product} if $Q$ is of the form
$$
Q(z)=v\prod_{m=1}^M (b_m(z)R_m+(I-R_m)),
$$
where $v$ is an $n\times n$ unitary constant matrix, $b_m$ is a {\it
	Blaschke factor} of the form
$$
b_m(z)=\frac{z-\alpha_m}{1-\overline\alpha_m z}\ \
(\alpha_m\in\mathbb D)
$$
and $R_m$ is an orthogonal projection in $\mathbb  C^n$.

\section{Hyponormality of block Toeplitz operators with Trigonometric  polynomial symbol}
In this section, we examine certain properties of hyponormal block Toeplitz operators whose symbols are trigonometric polynomials. \ In the case of scalar-valued symbols, it has been observed that if the symbol $\varphi \in L^{\infty}$ of a Toeplitz operator $T_{\varphi}$ is a trigonometric polynomial, then the self-commutator $[T_{\varphi}^{\ast}, T_{\varphi}]$ is of finite rank. \ To see this, let $\varphi (e^{i\theta}) = \sum_{n=-k}^m a_n e^{in\theta} \in L^{\infty}$ be a trigonometric polynomial, where $k,m \in \mathbb{N} \cup \{0\}$. \ Recall that for $\varphi, \psi \in L^{\infty}$, the identity
$$T_{\varphi \psi} - T_\varphi T_\psi = H_{\overline{\varphi}}^{\ast} H_\psi$$
holds. \ Using this, we obtain 
$$
T_{\varphi}^{\ast} T_{\varphi} = T_{\overline{\varphi} \varphi} - H_{\varphi}^{\ast} H_{\varphi}\quad \text{and} \quad T_{\varphi} T_{\varphi}^{\ast} = T_{\varphi \overline{\varphi}} - H_{\overline{\varphi}}^{\ast} H_{\overline{\varphi}},
$$
which shows that
$$
[T_\varphi^{\ast}, T_\varphi] = T_{\overline{\varphi} \varphi} - T_{\varphi \overline{\varphi}} + H_{\overline{\varphi}}^{\ast} H_{\overline{\varphi}} - H_{\varphi}^{\ast} H_{\varphi} = H_{\overline{\varphi}}^{\ast} H_{\overline{\varphi}} - H_{\varphi}^{\ast} H_{\varphi}.
$$
Since $\varphi$ is a trigonometric polynomial, the Hankel operators $H_{\overline{\varphi}}^{\ast}, H_{\overline{\varphi}}, H_{\varphi}^{\ast}, H_{\varphi}$ are all of finite rank. \ Consequently, the self-commutator $[T_{\varphi}^{\ast}, T_{\varphi}]$ is a finite rank operator.

However, this result does not generally extend to the case of block Toeplitz operators. \ Consider the matrix-valued trigonometric polynomial symbol 
{\setlength{\belowdisplayskip}{0.35cm}
	\begin{eqnarray*}
		\Phi (e^{i\theta}) :=
		\begin{bmatrix}
			e^{i\theta} & 1\\
			0 & e^{-i\theta}
		\end{bmatrix} \in L^{\infty} (\mathbb{T}, M_2).
\end{eqnarray*}}%
In this case, the self-commutator $[T_{\Phi}^{\ast}, T_{\Phi}]$ of the block Toeplitz operator $T_\Phi$ is infinite dimensional. \ To illustrate this, observe the following calculation:
{\setlength{\belowdisplayskip}{0.3cm}
	\begin{align*}
		[T_{\Phi}^{\ast}, T_{\Phi}] &= 
		\begin{bmatrix}
			T_{e^{-i\theta}} & 0\\
			I & T_{e^{i\theta}}
		\end{bmatrix}
		\begin{bmatrix}
			T_{e^{i\theta}} & I\\
			0 & T_{e^{-i\theta}}
		\end{bmatrix} - 
		\begin{bmatrix}
			T_{e^{i\theta}} & I\\
			0 & T_{e^{-i\theta}}
		\end{bmatrix}
		\begin{bmatrix}
			T_{e^{-i\theta}} & 0\\
			I & T_{e^{i\theta}}
		\end{bmatrix}\\[12pt]
		& = 
		\begin{bmatrix}
			- T_{e^{i\theta}} T_{e^{-i\theta}} & T_{e^{-i\theta}} - T_{e^{i\theta}}\\
			T_{e^{i\theta}} - T_{e^{-i\theta}} & T_{e^{i\theta}} T_{e^{-i\theta}}
		\end{bmatrix}.
\end{align*}}%
Therefore, the range of  $[T_{\Phi}^{\ast}, T_{\Phi}]$ is infinite dimensional.

In general, the following identity holds for a trigonometric polynomial symbol $\Phi \in L^{\infty} ( \mathbb{T}, M_n):$
{\setlength{\belowdisplayskip}{0.1cm}
	\begin{eqnarray*}
		[T_{\Phi}^{\ast}, T_{\Phi}] = T_{\Phi^{\ast} \Phi - \Phi \Phi^{\ast}} + H_{\Phi^{\ast}}^{\ast} H_{\Phi^{\ast}} - H_{\Phi}^{\ast} H_{\Phi}.
\end{eqnarray*}}%
Hence, for a normal matrix-valued symbol $\Phi$, we obtain
{\setlength{\belowdisplayskip}{0.15cm}
	\begin{eqnarray*}
		[T_{\Phi}^{\ast}, T_{\Phi}] =  H_{\Phi^{\ast}}^{\ast} H_{\Phi^{\ast}} - H_{\Phi}^{\ast} H_{\Phi}.
\end{eqnarray*}}%
Since the Hankel operators $H_{\Phi^{\ast}}^{\ast}, H_{\Phi^{\ast}}, H_{\Phi}^{\ast}, H_{\Phi}$ are of finite rank, it follows that for a  normal matrix-valued trigonometric polynomial $\Phi$, the self-commutator $[T_{\Phi}^{\ast}, T_{\Phi}]$ is a finite rank operator. \ In particular, if $T_\Phi$ is hyponormal, with trigonometric polynomial symbol, then $[T_{\Phi}^{\ast}, T_{\Phi}]$ is a positive semi-definite finite rank operator. \ This observation motivates a detailed examination of hyponormal block Toeplitz operators with trigonometric polynomial symbols. 

C. Gu, J. Hendricks and D. Rutherford \cite{GHR} gave a characterization for hyponormal block Toeplitz operators. \ Also, they provided additional insights about the symbol of hyponormal block Toeplitz operators when the symbol is a matrix trigonometric polynomial \cite{GHR}. \ Here, we add some more properties of  hyponormal block Toeplitz operators with trigonometric polynomial symbol by developing the methods of R.E. Curto and W.Y. Lee \cite{CuLe}.

\begin{theorem}\label{C3S1TH2} 
	For $m, N \in \mathbb{N} \cup \{0\}$, suppose that $\Phi \in L^{\infty} (\mathbb{T}, M_n)$ is a trigonometric polynomial of the form $\Phi (e^{i\theta}) = \sum_{j=-m}^{N} A_j e^{ij\theta} I_n$, with $A_N$ is invertible. \ Then, the hyponormality of $T_\Phi$ is independent of the matrices $A_0, A_1, \ldots A_{N-m}$.
\end{theorem}

\begin{proof}
	Since $T_\Phi$ is hyponormal, it follows from \cite[Corollary 5.2]{GHR} that $m \leq N$. \ Moreover, there exists a function $K \in H^{\infty} (\mathbb{T}, M_n)$ with $\|K\|_{\infty} \leq 1$ such that
	{\setlength{\belowdisplayskip}{0.1cm}
		\begin{eqnarray*}
			\Phi - K\Phi^{\ast} \in H^{\infty} (\mathbb{T}, M_n).
	\end{eqnarray*}}%
	Consequently, we obtain
	{\setlength{\belowdisplayskip}{0.35cm}
		\begin{eqnarray*}
			(K \Phi_{+}^{\ast} - \Phi_{-}^{\ast}) (e^{i\theta}) = K \sum\limits_{j=1}^{N} A_j^{\ast} e^{-ij\theta} I_n - \sum\limits_{j=1}^{m} A_{-j} e^{-ij\theta} I_n \in H^{\infty} (\mathbb{T}, M_n).
	\end{eqnarray*}}%
	Let $K (e^{i\theta}) = \sum_{j=0}^{\infty} K_j e^{ij\theta} I_n$.
	Then, we get
	{\setlength{\belowdisplayskip}{0.15cm}
		\begin{eqnarray*}
			K_0 = K_1 = \cdots = K_{N-m-1} = 0I_n \quad \text{and} \quad K_{N-m} = A_{-m}(A_N^{\ast})^{-1}.
	\end{eqnarray*}}%
	Moreover, for $n = N-m+1, N-m+2, \ldots, N-1$, we have
	{\setlength{\belowdisplayskip}{0.35cm}
		\begin{eqnarray*}
			K_n = \Bigg(A_{-N + n} - \sum\limits_{j=N-m}^{n-1} K_j A_{N-n+j}^{\ast}\Bigg){(A_N^{\ast})}^{-1}.
	\end{eqnarray*}}%
	To see this, observe that 
	{\setlength{\belowdisplayskip}{0.15cm}
		\begin{eqnarray*}
			K_j e^{ij\theta} I_n \cdot A_N^{\ast} e^{-iN \theta} I_n = K_j A_N^{\ast} e^{i(j-N)\theta} I_n.
	\end{eqnarray*}}%
	If $j <N-m$, then $K_j$ must be zero; otherwise, the term $K_j A_N^{\ast} e^{i(j-N)\theta} I_n$ will appear in the expansion of $(K \Phi_{+}^{\ast} - \Phi_{-}^{\ast}),$ which contradicts its analyticity. \ In the case where $j =N -m$, we observe that
	{\setlength{\belowdisplayskip}{0.15cm}
		\begin{eqnarray*}
			(K_{N-m} A_N^{\ast} - A_{-m}) e^{-im\theta} I_n = 0 \quad \text{if and only if}\quad  K_{N-m} = A_{-m}(A_N^{\ast})^{-1}.
	\end{eqnarray*}}%
	
	For $j > N-m,$ we obtain the recurrence relation
	{\setlength{\belowdisplayskip}{0.3cm}
		\begin{eqnarray*}
			K_j {A_N^{\ast}} + \sum\limits_{i=N-m}^{j-1} K_i A_{N-j+i}^{\ast}  - A_{-N + j} = 0.
	\end{eqnarray*}}%
	This relation is obtained by equating the corresponding Fourier coefficients in the expansion of $(K \Phi_{+}^{\ast} - \Phi_{-}^{\ast}).$ 
	
	Since $N-n+j > N-m$ for $n = N-m+1, N-m+2, \ldots, N-1$ and $N-m \leq j \leq n-1$, the recurrence is well-defined and the proof is complete.
\end{proof}


\begin{theorem}\label{C3S1TH3}
	Suppose that $\Phi = F^{\ast} + G = \sum_{j=-m}^{N}A_j e^{ij\theta} I_n \in L^{\infty} (\mathbb{T}, M_n),$ where $F$ and $G$ are analytic polynomials of degree $m$ and $N$, respectively and the leading coefficient $A_N$ is invertible. \ Now, define $\Psi = F^{\ast} + \mathbb{P} (e^{-ir\theta}I_n G)$, where $r\leq N-m$ and $\mathbb{P} (A)$ denotes the analytic part of the matrix-valued function $A \in L^{\infty}(\mathbb{T}, M_n)$. \ Then, $T_\Psi$ is hyponormal if and only if $T_\Phi$ is hyponormal. 
\end{theorem}


\begin{proof}
	Let $G_0 (e^{i\theta}) = \sum_{j=r}^{N}A_j e^{ij\theta}$ and define $\Phi_0 := F^{\ast} + G_0$. \ Consider the function $H (e^{i\theta}) = \mathbb{P} (e^{-ir\theta} I_n G (e^{i\theta})).$ Then, $H (e^{i\theta}) = e^{-ir\theta} I_n G_0 (e^{i\theta})$ and $\Psi = F^{\ast} + H$. \ By Theorem \ref{C3S1TH2}, $T_\Phi$ is hyponormal if and only if $T_{\Phi_0}$ is hyponormal. 
	
	Let us first assume that the operator $T_\Phi$ is hyponormal. \ Since $T_{\Phi_0}$ is also hyponormal, there exists $K \in H^{\infty} (\mathbb{T},M_n)$ such that 
	{\setlength{\belowdisplayskip}{0.15cm}
		\begin{eqnarray*}
			\Phi_0 - K\Phi_0^{\ast} \in H^{\infty}(\mathbb{T},M_n).
	\end{eqnarray*}}%
	Equivalently, this condition can be expressed as
	{\setlength{\belowdisplayskip}{0.15cm}
		\begin{eqnarray*}
			{\Phi_{0}}_{-}^{\ast} - K{\Phi_{0}}_{+}^{\ast} = F^{\ast} - K G_0^{\ast} \in H^{\infty}(\mathbb{T},M_n).
	\end{eqnarray*}}%
	Furthermore, by Theorem \ref{C3S1TH2}, the function $K$ admits a Fourier expansion of the form
	{\setlength{\belowdisplayskip}{0.25cm}
		\begin{eqnarray*}
			K (e^{i\theta}) = \sum_{j=N-m}^{\infty} C_j e^{ij\theta} I_n.
	\end{eqnarray*}}%
	Since $G^{\ast}_0 (e^{i\theta}) = e^{-ir\theta} I_n H^{\ast} (e^{i\theta})$, it follows that
	{\setlength{\belowdisplayskip}{0.15cm} 
		\begin{eqnarray*}
			F^{\ast} - K e^{-ir\theta} I_n H^{\ast} \in H^{\infty} (\mathbb{T}, M_n).
	\end{eqnarray*}}%
	Under the assumption $r\leq N-m$, we define $K^{\prime} := e^{-ir\theta} I_n  K.$ Then, $K^{\prime}$ belongs to $H^{\infty} (\mathbb{T}, M_n)$ and we see that
	{\setlength{\belowdisplayskip}{0.15cm} 
		\begin{eqnarray*}
			\|K^{\prime}\|_{\infty} = \|K e^{-ir\theta} I_n \|_{\infty} \leq \|K\|_{\infty} \leq 1.
	\end{eqnarray*}}%
	Hence,
	{\setlength{\belowdisplayskip}{0.15cm} 
		\begin{eqnarray*}
			F^{\ast} - K^{\prime} H^{\ast} = \Psi_{-}^{\ast} - K^{\prime} \Psi_{+}^{\ast} \in H^{\infty} (\mathbb{T}, M_n),
	\end{eqnarray*}}%
	which implies that 
	{\setlength{\belowdisplayskip}{0.15cm}
		\begin{eqnarray*}
			\Psi - K^{\prime} \Psi^{\ast} \in H^{\infty} (\mathbb{T}, M_n).
	\end{eqnarray*}}%
	Therefore, $T_\Psi$ is hyponormal.
	
	Conversely, assume that $T_\Psi$ is hyponormal. \ Then, there exists a matrix-valued function $K^{\prime} \in H^{\infty} (\mathbb{T}, M_n)$ with $\|K^{\prime}\|_{\infty} \leq 1$ such that 
	{\setlength{\belowdisplayskip}{0.1cm}
		\begin{eqnarray*}
			\Psi - K^{\prime} \Psi^{\ast} \in H^{\infty} (\mathbb{T}, M_n).
	\end{eqnarray*}}%
	Equivalently, since $\Psi_{-} = F$ and $\Psi_{+} = H$, it follows that
	{\setlength{\belowdisplayskip}{0.1cm}
		\begin{eqnarray}\label{C3EQ1}
			\Psi_{-}^{\ast} - K^{\prime} \Psi_{+}^{\ast} = F^{\ast} - K^{\prime} H^{\ast} \in H^{\infty} (\mathbb{T}, M_n).
	\end{eqnarray}}%
	Now, we define $K = K^{\prime} e^{ir\theta} I_n$. \ Clearly, $K \in H^{\infty} (\mathbb{T}, M_n).$ Moreover, $\|e^{ir\theta}\|_{\infty}\leq 1$ implies that $\|K\|_\infty \leq 1.$ From Equation \eqref{C3EQ1}, it follows that 
	{\setlength{\belowdisplayskip}{0.1cm} 
		\begin{eqnarray*} 
			F^{\ast} - K e^{-ir\theta} I_n H^{\ast} \in H^{\infty} (\mathbb{T}, M_n).
	\end{eqnarray*}}%
	Since $G^{\ast}_0 = e^{-ir\theta} I_n H^{\ast},$ we immediately obtain $F^{\ast} - K G_0^{\ast} \in H^{\infty} (\mathbb{T}, M_n).$ The representation $\Phi_0 = F^{\ast} + G_0$ implies that 
	${\Phi_{0}}_{-}^{\ast} - K{\Phi_{0}}_{+}^{\ast} \in H^{\infty} (\mathbb{T}, M_n).$ Equivalently,
		\begin{eqnarray*}
			\Phi_0 - K\Phi_0^{\ast} \in H^{\infty} (\mathbb{T}, M_n).
	\end{eqnarray*}
	Hence $T_{\Phi_0}$ is hyponormal and therefore, $T_\Phi$ is hyponormal.
\end{proof}

\section{Hyponormal and subnormal block Toeplitz operators with finite rank self-commutators}

In this section, we identify certain hyponormal block Toeplitz operators with finite rank self-commutator by analyzing the Cowen set of the symbol. \ T. Nakazi and K. Takahashi \cite{NT} proved that if the set $\mathcal{E}(\varphi)$ contains at least two elements for some $\varphi \in L^{\infty}(\mathbb{T})$, then $\varphi$ must be of bounded type. \ However, this result does not extend to the matrix-valued case, where $\Phi \in L^{\infty} (\mathbb{T}, M_n)$. \ For example, consider the matrix-valued function 
\begin{eqnarray*}
	\Phi (e^{i\theta}) :=
	\begin{bmatrix}
		(f + \overline{f}) (e^{i\theta}) & 0\\
		0 & e^{i\theta}
	\end{bmatrix} \in L^{\infty} (\mathbb{T}, M_2),
\end{eqnarray*}
where $\overline{f}$ is not of bounded type. \ Since 
\begin{eqnarray*}
	\Phi (e^{i\theta}) = 
	\begin{bmatrix}
		1 & 0\\
		0 & e^{i\theta}
	\end{bmatrix}
	\begin{bmatrix}
		(f + \overline{f}) (e^{i\theta}) & 0\\
		0 & e^{-i\theta}
	\end{bmatrix} + 
	\begin{bmatrix}
		0 & 0 \\
		0 & e^{i\theta} - 1
	\end{bmatrix}
\end{eqnarray*}
and 
\begin{eqnarray*}
	\Phi (e^{i\theta}) = 
	\begin{bmatrix}
		1 & 0\\
		0 & e^{i2\theta}
	\end{bmatrix}
	\begin{bmatrix}
		(f + \overline{f}) (e^{i\theta}) & 0\\
		0 & e^{-i\theta}
	\end{bmatrix},
\end{eqnarray*}
it follows that
\begin{eqnarray*}
	B (e^{i\theta}) = 
	\begin{bmatrix}
		1 & 0\\
		0 & e^{i\theta}
	\end{bmatrix} \quad \text{and} \quad 
	B_0 (e^{i\theta}) = 
	\begin{bmatrix}
		1 & 0\\
		0 & e^{i2\theta}
	\end{bmatrix}
\end{eqnarray*}
are elements of $\mathcal{E}(\Phi)$. \ Nevertheless, $\Phi$ is not of bounded type.

We now turn our attention to the following Conjecture \ref{C3S2CON1} formulated by R.E. Curto, I.S. Hwang and W.Y. Lee \cite{CHL1}, which plays an important role in the study of hyponormal block Toeplitz operators with finite rank self-commutator.

\begin{conjecture}\cite[Conjecture 6.1]{CHL1}\label{C3S2CON1}
	If $\Phi \in L^{\infty}(\mathbb{T}, M_n)$ is such that $T_\Phi$ is a hyponormal operator whose self-commutator $[T_\Phi^{\ast}, T_\Phi]$ is of finite rank, then there exists a finite Blaschke-Potapov product $B \in \mathcal{E}(\Phi)$ such that $\text{Rank }[T_\Phi^{\ast}, T_\Phi] = \text{deg }(\det B)$. 
\end{conjecture}


M. Abhinand, R.E. Curto, I.S. Hwang, W.Y. Lee and T. Prasad \cite{ACHLP} proved that if $\mathcal{E}(\Phi)$ contains a finite Blaschke-Potapov product $Q$, for a normal symbol $\Phi \in L^{\infty}(\mathbb{T}, M_n)$, then $[T_\Phi^{\ast}, T_\Phi]$ is a finite rank operator. \ At first, we check whether the conjecture is valid for normal block Toeplitz operator. \ Recall the following characterization of normal block Toeplitz operators by C. Gu, J. Hendricks and D. Rutherford \cite{GHR}:

\begin{lemma}\cite[Theorem 4.3]{GHR} \label{C3S2LE3} 
	Let $G= G_{+^{\prime}} + G_0 + G_{-}^{\ast} \in L^{\infty} (\mathbb{T}, M_n),$ where $G_0$ is a constant matrix and $G_{+^{\prime}} = G_{+} - G_0$. \ Suppose that $\det G_{+^{\prime}}$ is not identically zero. \ Then $T_G$ is normal if and only if the following two conditions are satisfied:
	\begin{enumerate}
		\item[(i)] $G^{\ast} G = G G^{\ast}$ almost everywhere on $\mathbb{T}$ and 
		
		\item[(ii)] there exists a constant unitary matrix $U \in M_n$ such that $G_{+^{\prime}} = G_{-} U$.
	\end{enumerate}
\end{lemma}


Next, we show that if $\mathcal{E}(\Phi)$ contains a constant unitary matrix $U$, then $T_\Phi$ is normal by modifying the methods and techniques in \cite{ACHLP}.

\begin{theorem} \label{C3S2TH4} 
	If $\Phi \in L^{\infty}(\mathbb{T}, M_n)$ and $\mathcal{E}(\Phi)$ contains a constant unitary matrix $U$, then the block Toeplitz operator $T_\Phi$ is normal.
\end{theorem}


\begin{proof}
	Suppose that $U \in \mathcal{E} (\Phi)$ is a constant unitary matrix, i.e., $U^{\ast} U = U U^{\ast} = I_n.$
	Then, it follows that
	{\setlength{\belowdisplayskip}{0.2cm}
		\begin{eqnarray*}
			[T_\Phi^{\ast}, T_\Phi] = H_{\Phi^{\ast}}^{\ast} (I - T_{\widetilde{U}} T_{\widetilde{U}^{\ast}}) H_{\Phi^{\ast}} = 0.
	\end{eqnarray*}}%
	This establishes the normality of $T_\Phi$ and thus completes the proof.
\end{proof}


The following theorem addresses the converse direction of Theorem \ref{C3S2TH4} by providing conditions under which the normality of the block Toeplitz operator $T_\Phi$ implies the existence of a constant unitary matrix in $\mathcal{E} (\Phi).$

\begin{theorem}\label{C3S2TH5}
	Let $\Phi= \Phi_{+^{\prime}} + \Phi_0 + \Phi_{-}^{\ast} \in L^{\infty}(\mathbb{T}, M_n)$ be such that $\det \Phi_{+^{\prime}}$ is not identically zero. \ If $T_\Phi$ is normal, then $\mathcal{E}(\Phi)$ contains a  constant unitary matrix $U \in M_n.$ 
\end{theorem}


\begin{proof}
	Recall that every normal operator is hyponormal. \ Therefore, the normality of $T_\Phi$ ensures that it is hyponormal. \ This, in turn, guarantees the existence of  a function $K \in H^{\infty}(\mathbb{T}, M_n)$ with $\|K\|_{\infty} \leq 1$ such that
	{\setlength{\belowdisplayskip}{0.1cm} 
		\begin{eqnarray*}
			\Phi - K\Phi^{\ast} \in H^{\infty}(\mathbb{T}, M_n).
	\end{eqnarray*}}%
	Equivalently, this condition holds if and only if 
	{\setlength{\belowdisplayskip}{0.1cm}
		\begin{eqnarray*}
			\Phi_{-}^{\ast} - K \Phi_{+}^{\ast} \in H^{\infty}(\mathbb{T}, M_n).
	\end{eqnarray*}}%
	By applying Lemma \ref{C3S2LE3}, we conclude that the normality of $T_\Phi$ implies two conditions: 
	{\setlength{\belowdisplayskip}{0.1cm}
		\begin{eqnarray*}
			\Phi^{\ast} \Phi = \Phi \Phi^{\ast} \quad \text{and} \quad \Phi_{+^{\prime}} = \Phi_{-} U
	\end{eqnarray*}}%
	for some constant unitary matrix $U \in M_n$. \ To verify that $U \in \mathcal{E} (\Phi)$, observe that
	{\setlength{\belowdisplayskip}{0.1cm}
		\begin{align*}
			\Phi_{-}^{\ast} - U \Phi_{+}^{\ast} & = \Phi_{-}^{\ast} - U ({\Phi_{+^{\prime}}}^{\ast} + \Phi_0^{\ast})\\
			& = \Phi_{-}^{\ast} - U (U^{\ast} \Phi_{-}^{\ast} + \Phi_0^{\ast})\\
			& = - U \Phi_0^{\ast},
	\end{align*}}%
	where the last equality holds because $U$ is unitary. \ Hence, $\Phi_{-}^{\ast} - U \Phi_{+}^{\ast}$ belongs to $H^{\infty}(\mathbb{T}, M_n)$, which completes the proof.
\end{proof}

The following example shows that $\mathcal{E}(\Phi)$ contains a constant unitary matrix even if $\det \Phi_{+} = 0.$

\begin{example}\label{C3S2EX6}
	Consider the function 
	\begin{eqnarray*}
		\Phi (e^{i\theta}) := 
		\begin{bmatrix}
			e^{i\theta} + e^{-i\theta} & e^{i\theta} + e^{-i\theta} \\
			e^{i\theta} + e^{-i\theta} & e^{i\theta} + e^{-i\theta}
		\end{bmatrix} \in L^{\infty} (\mathbb{T}, M_2).
	\end{eqnarray*}
	The associated block Toeplitz operator is given by 
	\begin{eqnarray*}
		T_\Phi = 
		\begin{bmatrix}
			T_{e^{i\theta} + e^{-i\theta}} & T_{e^{i\theta} + e^{-i\theta}}\\
			T_{e^{i\theta} + e^{-i\theta}} & T_{e^{i\theta} + e^{-i\theta}}
		\end{bmatrix}.
	\end{eqnarray*}
	Note that $T_\Phi$ is self-adjoint and hence normal. \ Observe that the symbol $\Phi$ satisfies the relation $\Phi = I_2 \Phi^{\ast},$ where $I_2$ is the $2\times 2$ identity matrix. \ Consequently, the constant unitary matrix $U = I_2$ belongs to $\mathcal{E} (\Phi).$ However, a straightforward computation shows that
	\begin{eqnarray*}
		\Phi_{+^{\prime}} = 
		\begin{bmatrix}
			e^{i\theta} & e^{i\theta}\\
			e^{i\theta} & e^{i\theta}
		\end{bmatrix} \quad \text{and hence } \det \Phi_{+^{\prime}} \equiv 0.
	\end{eqnarray*}
\end{example}

Now, we give a partial answer to Conjecture \ref{C3S2CON1}.

\begin{theorem}\label{C3S2TH7}
	Let $\Phi \in L^{\infty}(\mathbb{T}, M_n)$ be such that each row of $\Phi^{\ast}$ contains at least one scalar-valued function that is not of bounded type. \ If $T_\Phi$ is hyponormal and has a self-commutator of finite rank, then the set $\mathcal{E}(\Phi)$ contains a finite Blaschke-Potapov product. 
\end{theorem}


\begin{proof}
Assume that each row of $\Phi^{\ast}$ contains at least one scalar-valued function that is not of bounded type. \ Under this condition, it follows that the Hankel operator  $H_{\Phi^{\ast}}$ has dense range. \ Consequently, the adjoint operator $H_{\Phi^{\ast}}^{\ast}$ is injective. \ Therefore by \cite[Lemma 4.1]{ACHLP}, we obtain the identity
\begin{eqnarray*}
\operatorname{Rank }[T_\Phi^{\ast}, T_\Phi] = \operatorname{Rank } (I - T_{\widetilde{K}} T_{\widetilde{K}^{\ast}} ),
\end{eqnarray*}
for some $K \in \mathcal{E} (\Phi).$ Since the commutator  $[T_\Phi^{\ast}, T_\Phi]$ is assumed to be of finite rank, it follows that the operator $I - T_{\widetilde{K}} T_{\widetilde{K}^{\ast}}$ is also of finite rank. \ This implies that $\widetilde{K}$ must be a finite Blaschke-Potapov product and hence so is $K$. \ This completes the proof.
\end{proof}

The following example shows that $\mathcal{E}(\Phi)$ contains a finite Blaschke-Potapov product even if $\Phi^{\ast}$ has a row that doesn't contain a function which is not of bounded type.

\begin{example} \label{C3S2EX8} 
	Consider the matrix-valued function
	\begin{eqnarray*}
		\Phi (e^{i\theta}) := 
		\begin{bmatrix}
			(f+\overline{f}) (e^{i\theta}) & 0\\
			0 & e^{i\theta}
		\end{bmatrix} \in L^{\infty} (\mathbb{T}, M_2),
	\end{eqnarray*}
	where $f \in H^2$. \ In this case, the block Toeplitz operator $T_\Phi$ is hyponormal and the self-commutator is given by
	\begin{eqnarray*}
		[T_\Phi^{\ast}, T_\Phi] = 
		\begin{bmatrix}
			0 & 0\\
			0 & [T_{e^{i\theta}}^{\ast}, T_{e^{i\theta}}]
		\end{bmatrix}.
	\end{eqnarray*}
	Since $[T_{e^{i\theta}}^{\ast}, T_{e^{i\theta}}]$ has rank one, it follows that the self-commutator $[T_\Phi^{\ast}, T_\Phi]$ is a finite rank operator . \ However, note that the second row of $\Phi$ does not contain a scalar-valued function that fails to be of bounded type. \ Despite this, the matrix-valued function
	\begin{eqnarray*}
		B (e^{i\theta}) = 
		\begin{bmatrix}
			1 & 0\\
			0 & e^{i\theta}
		\end{bmatrix} \in \mathcal{E}(\Phi).
	\end{eqnarray*}
\end{example}

\begin{theorem} \label{thm38}
Let $\Phi \in H^{\infty}(\mathbb{T}, M_n)$ be such that $\Phi^{\ast}$ is of bounded type and $T_\Phi$ is hyponormal. \ Then $[T_\Phi^{\ast}, T_\Phi]$ is a finite rank operator if and only if there exists a finite Blaschke–Potapov product in $\mathcal{E}(\widetilde{\Phi})$. 
\end{theorem}

\begin{proof}
Recall that 
$$[T_\Phi^{\ast}, T_\Phi] = H_{\Phi^{\ast}}^{\ast} H_{\Phi^{\ast}} - H_\Phi^{\ast} H_\Phi + T_{[\Phi^{\ast}, \Phi]}.$$
Since $\Phi$ is analytic, it follows that $H_\Phi \equiv 0$ and hence, $H_\Phi^{\ast} H_\Phi \equiv 0.$ Hyponormality of $T_\Phi$ implies that $\Phi (e^{i\theta})$ is normal almost everywhere \cite[Theorem 3.3]{GHR}. \ That is, $T_{[\Phi^{\ast},\Phi]} \equiv 0.$ Therefore, we obtain $$[T_\Phi^{\ast}, T_\Phi] = H_{\Phi^{\ast}}^{\ast} H_{\Phi^{\ast}}.$$

At first, assume that $[T_\Phi^{\ast}, T_\Phi]$ is a finite rank operator. \ Because $\Phi^{\ast}$ is of bounded type, it follows that
$$\ker [T_\Phi^{\ast}, T_\Phi] = \ker H_{\Phi^{\ast}}^{\ast} H_{\Phi^{\ast}} = \ker H_{\Phi^{\ast}} = \Theta H^2(\mathbb{T}, \mathbb{C}^n),$$
for some inner function $\Theta \in H^{\infty} (\mathbb{T}, M_n)$ \cite[Theorem 2.2]{GHR}. \ Consequently,
$$0 = H_{\Phi^{\ast}} \Theta H^2(\mathbb{T}, \mathbb{C}^n) = H_{\Phi^{\ast}} T_\Theta H^2(\mathbb{T}, \mathbb{C}^n) = H_{\Phi^{\ast} \Theta} H^2(\mathbb{T}, \mathbb{C}^n),$$
which implies that $\Phi^{\ast} \Theta \in H^{\infty} (\mathbb{T}, M_n).$ Since $\widetilde{\Phi}$ and $\widetilde{\Theta} \widetilde{\Phi^{\ast}}$  are analytic, it follows that $\widetilde{\Phi} - \widetilde{\Theta} \widetilde{\Phi^{\ast}} \in H^{\infty} (\mathbb{T}, M_n).$ Since $\ker [T_\Phi^{\ast}, T_\Phi] = \Theta H^2(\mathbb{T}, \mathbb{C}^n)$ and $\operatorname{Ran} [T_\Phi^{\ast}, T_\Phi]$ is finite dimensional, we obtain
$$\mathcal{H} (\Theta) = \overline{\operatorname{Ran} [T_\Phi^{\ast}, T_\Phi]} = \operatorname{Ran} [T_\Phi^{\ast}, T_\Phi] < \infty.$$
By \cite[Lemma 2.4]{CHL}, this forces $\Theta$ to be a finite Blaschke–Potapov product, and hence so is $\widetilde{\Theta}$. 

Conversely, assume that there exists a finite Blaschke–Potapov product $\widetilde{\Theta} \in \mathcal{E} (\widetilde{\Phi}).$ That is $\widetilde{\Phi} - \widetilde{\Theta} \widetilde{\Phi^{\ast}} \in H^{\infty} (\mathbb{T}, M_n).$ Since $\Phi$ is analytic, $\widetilde{\Phi}$ is also analytic and so is $\widetilde{\Theta} \widetilde{\Phi^{\ast}}.$ This implies that $\Phi^{\ast} \Theta \in H^{\infty} (\mathbb{T}, M_n).$ Therefore, we observe that
$$0 = H_{\Phi^{\ast} \Theta} H^2(\mathbb{T}, \mathbb{C}^n) = H_{\Phi^{\ast}} T_\Theta H^2(\mathbb{T}, \mathbb{C}^n)= H_{\Phi^{\ast}} \Theta H^2(\mathbb{T}, \mathbb{C}^n).$$
Consequently, $\Theta H^2(\mathbb{T}, \mathbb{C}^n) \subseteq \ker H_{\Phi^{\ast}} =  \ker H_{\Phi^{\ast}}^{\ast} H_{\Phi^{\ast}} = \ker [T_\Phi^{\ast}, T_\Phi].$ This implies that $(\ker [T_\Phi^{\ast}, T_\Phi])^{\perp} = \overline{\operatorname{Ran} [T_\Phi^{\ast}, T_\Phi]} \subseteq \mathcal{H} (\Theta).$ Since $\widetilde{\Theta}$ is a finite Blaschke-Potapov product, it follows that $\Theta$ is also a finite Blaschke-Potapov product. \ Consequently, $\dim \mathcal{H}(\Theta) < \infty$ \cite[Lemma 2.4]{CHL}. \ That is, 
$\dim \operatorname{Ran} [T_\Phi^{\ast}, T_\Phi] \leq \dim \overline{\operatorname{Ran} [T_\Phi^{\ast}, T_\Phi]} \leq \dim \mathcal{H}(\Theta) < \infty.$
\end{proof}

\begin{remark}
A careful analysis of the second part of the Proof of Theorem \ref{thm38} shows that if $\mathcal{E}(\widetilde{\bar{\Phi}})$ contains a finite Blaschke-Potapov product , for a symbol $\Phi \in H^\infty(\mathbb{T},M_n)$, then $[T_\Phi^*,T_\Phi)]$ is a finite rank operator.
\end{remark}

 
\section{Open Problems}
In the context of Toeplitz operators, the hyponormal Toeplitz operators with finite rank self-commutator were characterized by T. Nakazi and K. Takahashi \cite{NT}. \ Building upon this direction of research, R.E. Curto, I.S. Hwang and W.Y. Lee \cite{CHL1} examined the analogous question in the framework of block Toeplitz operators and proposed Conjecture \ref{C3S2CON1}. \ Theorems \ref{C3S2TH4}, \ref{C3S2TH5}, \ref{C3S2TH7} and \cite[Lemma 4.3]{ACHLP} provide partial progress towards solving this conjecture.
 
Despite these recent developments, a complete solution of Conjecture \ref{C3S2CON1} remains open. \ Thus, it is appropriate to re-frame the conjecture in light of our partial results. \ The following problems are proposed as refined versions of the conjecture, whose answers would significantly contribute to a complete characterization of block Toeplitz operators with finite rank self-commutator.
 
\begin{problem}
Let $\Phi = \Phi_{+^{\prime}} + \Phi_0 + \Phi_{-}^{\ast} \in L^{\infty}(\mathbb{T}, M_n)$ be such that $\det \Phi_{+^\prime} \equiv 0$ and $T_\Phi$ is normal. \ Does there exist a finite Blaschke-Potapov product $B \in \mathcal{E}(\Phi)$ such that $\deg(\det B) = 0$?
\end{problem}
 
\begin{problem}
Let $\Phi \in L^{\infty}(\mathbb{T}, M_n)$ be a normal symbol and suppose that $\mathcal{E}(\Phi)$ contains a finite Blaschke-Potapov product $B$. \ Does it follow that 
 \begin{eqnarray*}
 \mathrm{Rank }[T_\Phi^{\ast}, T_\Phi] = \deg(\det B)?
 \end{eqnarray*}
 \end{problem}
 Motivated by Example \ref{C3S2EX8}, we pose the following problem:
 
 \begin{problem}
 Suppose $\Phi^{\ast}$ is not of bounded type such that there exists a row of $\Phi^{\ast}$ that contains no scalar-valued function which is not of bounded type. \ If $T_\Phi$ is a hyponormal block Toeplitz operator with finite rank self-commutator, does there exist a finite Blaschke-Potapov product $B \in \mathcal{E}(\Phi)$ such that
 \begin{eqnarray*}
 \deg(\det B) = \mathrm{Rank}[T_\Phi^{\ast}, T_\Phi]?
 \end{eqnarray*}
 \end{problem}
 
 \begin{problem}
 Suppose that $\Phi \in L^{\infty}(\mathbb{T}, M_n)$ is such that $T_\Phi$ is a hyponormal operator with finite rank self-commutator $[T_\Phi^{\ast}, T_\Phi]$. \ If $\Phi^{\ast}$ is of bounded type, does there exist a finite Blaschke-Potapov product $B \in \mathcal{E}(\Phi)$ such that
 \begin{eqnarray*}
 \mathrm{Rank}[T_\Phi^{\ast}, T_\Phi] = \deg(\det B)?
 \end{eqnarray*}
 \end{problem}

\vskip 1cm

\section{Declarations}
\subsection{Funding} 
M. Abhinand is supported by the Junior Research Fellowship, University Grant commission, Government of India. \ R. Curto is partially supported by U.S. National Science Foundation grant DMS-2247167. \  T. Prasad is supported in part by the Mathematical Research Impact Centric Support, MATRICS (MTR/2021/000373) by SERB, Department of Science and Technology (DST), Government of India.

\subsection{Conflicts of interest/competing interests} 

{\bf Non-financial interests}: \ The authors have no competing interests to declare that are relevant to the content of this article.

\subsection{Data availability.}
All data generated or analyzed during this study are included in this article.


\end{document}